\documentclass{conm-p-l}
\usepackage{amsmath,enumerate,amsthm,amssymb,amscd}

\input{xy}\xyoption{all}
\CompileMatrices

\newcommand{\gothic}{\mathfrak}
\newcommand{\ra}{\rightarrow}

\newcommand{\m}{{\gothic{m}}}

\newcommand{\Tor}{\operatorname{Tor{}}}

\renewcommand{\phi}{\varphi}
\renewcommand{\to}{{\longrightarrow}}

\newcommand{\inc}{\subseteq}

\renewcommand{\phi}{\varphi}
\renewcommand{\to}{{\longrightarrow}}

\newtheorem{thm}{Theorem}[section]
\newtheorem{example}[thm]{Example}
\newtheorem{cor}[thm]{Corollary}
\newtheorem{prop}[thm]{Proposition}
\newtheorem{lemma}[thm]{Lemma}
\newtheorem{definition}[thm]{Definition}
\newtheorem{conj}[thm]{Conjecture}
\newtheorem{remark}[thm]{Remark}
\newtheorem{disc}[thm]{Discussion}

\numberwithin{equation}{section}

\begin{document}
\title{Absolute Integral Closure} 

\author{Craig Huneke}

\address{Department of Mathematics \\
        University of Kansas \\
        Lawrence, KS
        66045}

\email{huneke@math.ku.edu}

\urladdr{http://www.math.ku.edu/\textasciitilde huneke}

\date{Jan. 19, 2010}

\bibliographystyle{amsplain}

\numberwithin{thm}{section}
\begin{abstract} This paper is an expanded version of  three lectures the author
 gave during the summer school, ``PASI: Commutative Algebra and its Connections to Geometry,"
from August 3-14, 2009 held in Olinda, Brazil. 
\end{abstract}



\thanks{ The author was partially supported by the NSF, DMS-0756853.}
\maketitle

\tableofcontents
\def\addcontentsline#1#2#3{%
\addtocontents{#1}{\protect\contentsline{#2}{#3}{}}}

\vskip 20 pt

\section{Introduction}

\vskip 10 pt

This paper is based on three talks given at the PASI conference in Olinda, Brazil in the summer of
2009. One point of the paper is  to introduce students to  one aspect of characteristic $p$
methods in commutative algebra. Such methods have been among the most powerful in the
field. The basic method of reduction to
characteristic $p$ is used to prove results for arbitrary Noetherian rings containing a field; the field
can be characteristic $0$. Thus, even though one is often working in positive characteristic, ones main
interest might well be in rings containing the rationals. In the late 1980's, ``tight closure theory" was
discovered by M. Hochster and myself. This theory synthesized many existing reduction to characteristic $p$
proofs into one theory, which has now grown to encompass a large number of different directions.  In this paper,
we'll concentrate on a result which Karen Smith has refered to as a ``crown jewel" of
tight closure theory, namely the fact that the absolute integral closure of a complete local domain
of positive characteristic is Cohen-Macaulay. We first introduce basic concepts, and discuss some rather
amazing properties of absolute integral closures. Then we prove some classical theorems in dimension one.
The main part of the paper will give a recent proof of the Cohen-Macaulayness of the absolute integral closure
in positive characteristic. Applications will be given in Section 6. Some further thoughts
about directions to go will be presented in the last section. 
For unexplained results or definitions, we refer the reader to Eisenbud's book \cite{Ei}.
\section{Basic Concepts}

We begin with a definition:

\begin{definition} {\rm Let $R$ be a domain, and let $R\subset S$, a ring. The \it integral closure \rm
of $R$ in $S$ is the set of all elements $s\in S$ which satisfy an equation of the form
$$s^n + r_1s^{n-1} + ... + r_n = 0,$$ where $r_i\in R$ for all $i$.}
\end{definition}

The set of elements in $S$ which are integral over $R$  form a ring, called the \it integral closure of
$R$ in $S$\rm.  When $S$ is the fraction field of $R$, the resulting ring is called the
integral closure of $R$. The ring $R$ is said to be \it integrally closed \rm if 
$R$ is equal to its integral closure. For example, a well-known result in commutative algebra says that
every UFD is integrally closed. So, e.g., the integers are integrally closed. Another
basic result is states that if $R$ is integrally closed, so is
the ring of polynomials $R[X]$ as well as the ring of formal power series $R[[X]]$.  An obvious induction shows that also
$R[X_1,...,X_n]$ and $R[[X_1,...,X_n]]$ are integrally closed for all $n\geq 1$. 

The main object we will study is given in the following definition.

\begin{definition} {\rm Let $R$ be an integral domain with fraction field $K$. Let $\overline{K}$ be a
fixed algebraic closure of $K$. The integral closure of $R$ in $\overline{K}$, denoted $R^+$, is
called the \it absolute integral closure \rm of $R$.}
\end{definition}

We simply say ``R-plus" to mean the absolute integral closure. This ring has been well-studied in
several different contexts. 

For other interesting results concerning $R^+$ which we will not be writing about, see
\cite{A},\cite{AH}, and \cite{Sm}.

\medskip

\medskip

\begin{disc} {\rm Whenever we study $R^+$, we might as well assume that $R$ is integrally closed itself; if $S$ is
the integral closure of $R$, then $R\subset S\subset R^+$, and $S^+ = R^+$. More generally,
if $S$ is \it any \rm finite integral extension of $R$ which is a domain, then the fraction field
of $S$ is algebraic over the fraction field of $R$, and hence there is an isomorphic copy of $S$ which
contains $R$ and sits inside $R^+$. By an abuse of language, we'll just assume that $S\subset R^+$. 
In this case, $S^+ = R^+$.  Studying $R^+$ is basically studying all integral extensions of $R$
at the same time.}
\end{disc}

\medskip

\begin{remark}  {\rm Continuing the discussion from above, if $R$ is local and complete in the $\m$-adic topology and contains
a field, then $R$ is actually a finite integral extension of a formal power series ring over a
field $k$. Namely, the Cohen structure theorem gives that there exists a field $k$ inside
$R$ such that the composition of maps $k\rightarrow R\rightarrow R/\m$ is an isomorphism.
Moreover, if $x_1,...,x_d$ is a system of parameters for $R$, then the complete local subring of $R$
generated by these elements, $k[[x_1,...,x_d]] = A$ is isomorphic with a formal power series ring
over $k$, and $R$ is module-finite over $A$.
Hence $A^+ = R^+$. 
Thus, if we fix a copy of the residue field, say $k$,  sitting inside $R$,  there is really only one $R^+$ for each dimension; it is the absolute integral closure of
the formal power series ring over  the field $k$. In other words, for each dimension, we are really just studying
one ring, the absolute integral closure of the power series ring over $k$. There is a classical theorem
which gives the structure of such a ring in dimension one over a field of characteristic $0$. See Theorem~\ref{newton}.  } 
\end{remark}

\medskip

 \bigskip

\begin{disc} {\rm It is not difficult to prove that if $R$ is an integral domain, then a domain $S$ integral
over $R$ is isomorphic with $R^+$ if and only if monic polynomials over $S$ factor into monic linear polynomials over $S$.
Using this criterion, one can easily prove two very basic properites of $R^+$. 

\medskip
$\bullet$ Firstly, if $q\in \text{Spec}(R)$, then $(R_q)^+\cong (R^+)_q$, where the latter localization is at the multiplicatively closed set
$R-q$.

\medskip

$\bullet$ Secondly, if $Q\in \text{Spec}(R^+)$, then $R^+/Q\cong (R/(Q\cap R))^+$.  Thus every quotient of $R^+$ by a prime ideal is itself
the absolute integral closure of an appropriate quotient of $R$.}
\end{disc}

We leave these statements as exercises below.

These properties not only allow us to mod out by primes and localize by changing the absolute integral closure
in a similar way, they will give very strong properites. For example, the main result of Section 4 states
that the absolute integral closure of an excellent local domain in positive characteristic is Cohen-Macaulay
in a suitable sense. The second property discussed above then says that the same is true modulo every prime
ideal of the absolute integral closure! This is remarkable.

M. Artin \cite{Ar} proved another amazing property of $R^+$:

\begin{thm} Let $R$ be a domain. If $P,Q$ are prime ideals in $R^+$, then either $P+Q = R^+$, or
$P+Q$ is prime.
\end{thm}

This is very far from what happens in rings we typically study. For example in the polynomial ring $k[x,y]$,
the ideals $P = (x+y^2)$ and $Q = (x)$ are prime, but the sum is not. The proof we give is taken from
\cite{HH}.

\begin{proof}
To prove Artin's result, let $ab\in P + Q$. Set $z= b-a$. Notice that $a^2 + za = ab = u+v$ for some $u\in P, v\in Q$.
The equation $X^2+zX= u$ has a solution $x\in R$. Since $u\in P$, it follows that either $x+z\in P$ or $x\in P$. But
$(x^2+zx)-(a^2+za) = u-(u+v)\in Q$, so either $x-a\in Q$ or $x+a+z\in Q$. In each of the four different cases, we
obtain that either $a\in P+Q$ or $b\in P+Q$. For example, if $x+z\in P$ and $x-a\in Q$, then $b= (x+z)-(x-a)\in Q$.
\end{proof}
\medskip

Notice that if $R$ is complete and local, then every finite extension domain of $R$ is also local. In particular,
$R^+$ is a local ring, i.e., has a unique maximal ideal. Therefore the possibility that $P+Q = R$ cannot happen
in this case, and the sum of every set of prime ideals is prime.

 \bigskip

\centerline{\bf Exercises}
\bigskip

1. Let $R$ be a domain with fraction field $K$, and let $L$ be an algebraic extension of $K$.
Prove that $L$ is algebraically closed if and only if every monic polynomial with coefficients in $R$
has a root in $L$.

\medskip

2. Let $R$ be an integral domain. Prove that a domain $S$ integral over $R$ is isomorphic with $R^+$
if and only if every monic polynomial over $S$ factors into monic linear polynomials over $S$.

\medskip

3. Prove that Artin's theorem can be extended to primary ideals: the sum of any two primary ideals in $R^+$
is again primary or the whole ring.
\medskip

4.  Let $R$ be a domain of positive dimension. Prove that $R^+$ is not Noetherian.

\medskip

5. Suppose that $R$ is an integral domain, and $R\inc S\inc R^+$. Prove that $R^+ = S^+$.

\medskip

6. If $q\in \text{Spec}(R)$, then $(R_q)^+\cong (R^+)_q$, where the latter localization is at the multiplicatively closed set
$R-q$.

\medskip

7. If $Q\in \text{Spec}(R^+)$, then $R^+/Q\cong (R/(Q\cap R))^+$.  Thus every quotient of $R^+$ by a prime ideal is itself
the absolute integral closure of an appropriate quotient of $R$.

\vskip.5truein

\section{Dimension One}
\bigskip

 Let $k$ be a field of characteristic $0$, and consider the formal power series ring $k[[t]]$, whose fraction field is
denoted $k((t))$. A famous theorem due to Puiseux, but apparently known to Newton, is that the algebraic closure of
the field $k((t))$ is the union of the fields $k((t^{1/n}))$ for $n\geq 1$.  A recent paper by Kedlaya (\cite{K}) gives
a characterization of this algebraic closure in the case $k$ has positive characteristic; it is quite hard even to describe.
Chevalley  \cite{C} pointed out that the Artin-Schreirer polynomial $x^p-x-t^{-1}$ has no root in the Newton-Puiseux field
$\bigcup_n k((t^{1/n}))$. 
In fact, as we shall see, there are remarkable differences in $R^+$ depending on the characteristic of the ground field.

We give a proof of the  Newton-Puiseux theorem, following a treatment in \cite{No}.

\begin{thm}\label{newton} Let $k$ be an algebraically closed field of characteristic $0$, and let $k((t))$ be the fraction field
of the formal power series ring $k[[t]]$. Then an algebraic closure of $k((t))$ is the Newton-Puiseux field $\cup_n k((t^{1/n}))$.
\end{thm}

\begin{proof} We need to prove that every monic polynomial $P(z)\in B[z]$ is reducible, where $B = \cup_nk[[t^{1/n}]]$.
Write $P = z^n+ a_1z^{n-1} + ... + a_n$, where $a_i\in B$.  By using a transformation $z' = z+\frac{1}{n}a_1$, we can
assume without loss of generality that $a_1 = 0$. This is called the ``Tschirhausen transformation". Furthermore, by replacing
$k[[t]]$ by $k[[t^{1/n}]]$ for some large $n$, we can change variables and assume that all $a_i\in k[[t]]$.   Let $r_k = \text{ord}(a_k)$,
so that $a_k = t^{r_k}u_k(t)$, where $u_k(t)$ is a unit. Here ``ord" denotes the $t$-adic order of an element. 

We want to factor $P(z)$, and to do so we will use Hensel's lemma\footnote{ The version of Hensel's lemma
we use is the following theorem:
\begin{thm} Let $(R,\m)$ be a local ring which is complete in the $\m$-adic topology. If $f(T)$ is a monic polynomial
with coefficients in $R$ such that $f(T)$ factors modulo $\m$ into the product of two relatively prime monic
polynomials of positive degree, then this factorization lifts to a factorization of $f(T)$ into two monic polynomials.
\end{thm}}; it suffices to factor $P(z)$ into two relatively prime polynomials
after going modulo $t$. This is always possible unless after reduction mod $t$, $P(z)$ becomes of the form $(t-\alpha)^n$ for
some $\alpha\in k$. However, since $a_1 = 0$, the only possible such $\alpha$ is $0$. Thus we are done unless for every $k$, $2\leq k\leq n$,
$a_i(0) = 0$.  We want to make a change of variables where this does not occur. 

Set $z = t^ry$, for some rational $r$ to be chosen later. Substituting, we see that
$P(z) = t^{rn}y^n + a_2t^{r(n-2)}y^{n-2} + ... + a_n$. We wish to factor out $t^{rn}$ from every term in such a way that
at least one term becomes a unit. The power of $t$ dividing
the $i$th term is $t^{r_i}t^{(n-i)r}$, so what we need to do is to choose $r$ in such a way that 
$r_i + (n-i)r \geq rn$, i.e., so that $r_i\geq ri$. We set $r = \text{min}\{\frac {r_i}{i}\}$. Then we can rewrite
$P(z) = t^{rn}Q(y,t)$, where $Q = y^n + b_2y^{n-1} + ... + b_n$. By again replacing $t$ by $t^{1/m}$ for suitably large $m$,
we can assume that each $b_j\in k[[t]]$. But now for at least one $b_i$, $b_i(0)\ne 0$; this occurs for every term where
$r_i/i = r$. It follows that we can factor $Q$ over the residue field into relatively prime polynomials, and by Hensel's Lemma
this lifts to a factorization of $Q$ and hence also of $P$. 
\end{proof}

This theorem has the following almost immediate corollary:

\begin{cor} Let $k$ be an algebraically closed field of characteristic $0$. Then $k[[t]]^+ = \cup_nk[[t^{1/n}]]$.
\end{cor}

\begin{proof} The ring on the right side of the equation is clearly integral over $k[[t]]$, and the fraction field
of it is clearly the Newton-Puiseux field, which is an algebraic closure of $k((t))$.
To finish the proof of the corollary, we need to prove that $\cup_nk[[t^{1/n}]]$ is integrally closed. But since
it is a union of discrete valuation rings, each integrally closed, it is also. \end{proof}

\begin{example} {\rm As an example, consider the
equation $X^2 - X = t^{-1}$. Using the quadratic formula and
Taylor series shows that the roots of this polynomial
are a power series in $t^{-1/2}$. One can solve recursively for the
coefficients. }
\end{example}
\medskip

The situation in positive characteristic is drastically different. We use the discussion in the  paper \cite{K} in what follows.
 A \it generalized power series \rm is an
expression of the form, $\sum_{i\in \mathbb Q} c_it^i$, with $c_i\in k$, such that the set of $i$ with $x_i\ne 0$
is a well-ordered subset of $\mathbb Q$, i.e., every non-empty subset has a least element. Such generalized power
series form a ring in a natural way; it makes sense to multiply and add them in the obvious way.
Abyhankar pointed out that with this generalization of the idea of a power series,
the Chevalley polynomial has the root $t^{-1/p} + t^{-1/p^2}+...$.

Are there enough generalized power series to obtain an algebraic closure of $k((t))$?  The following result was proved independently
by Huang \cite{Hua}, Rayner \cite{Ray}, and \c Stef\u anescu \cite{St}:

\begin{thm} Let $k$ be an algebraically closed field of positive characteristic, and let $K$ be the set of generalized power
 series of the form $f = \sum_{i\in S} c_it^i$,
with $c_i\in k$, where the set $S$ has the following properties:
\begin{enumerate}[\quad\rm (1)]
\item  $S$ is a subset of $\mathbb Q$ which depends on $f$.
\item  Every nonempty subset of $S$ has a least element.
\item  There exists a natural number $m$ such that every element of $mS$ has denominator a power of $p$.
\end{enumerate}
Then $K$ is an algebraically closed field containing $k((t))$.
\end{thm}

Kedlaya \cite{K} gives a construction  of the algebraic closure of $k((t))$ when $k$ is algebraically closed of positive characteristic.
By the above theorem, one needs to identify generalized power series which are algebraic over $k((t))$. 
His result is quite complicated to even describe, but a glimpse of the issues which arise can be seen
in the following result of Huang \cite{Hua} and \c Stef\u anescu \cite{St}:

\begin{thm} The series $\sum_{i = 0}^{\infty} c_it^{-1/p^i} $ with $c_i\in \overline{F_p}$, the
algebraic closure of the field with $p$ elements, is algebraic over $\overline{F_p}((t))$ if and only if 
the sequence $\{c_i\}$ is eventually periodic.
\end{thm}

A consequence of the main result of \cite{K} is the following nice theorem, which perhaps can be
proved directly.

\begin{thm} Let $k$ be an algebraically closed field of positive characteristic, and let $\sum_i c_it^i$ be a generalized power series which is algebraic over $k((t))$. Then for
every real number $\alpha$,  $\sum_{i < \alpha} c_it^i$ is also algebraic over $k((t))$.
\end{thm}

We now move away from the local case. 
Another classical result in dimension one concerns the ring of all algebraic integers, $\mathbb Z^+$.
To prove this theorem, we first recall some results about
class groups. 

Let $D$ be a ring of algebraic integers. This means that the fraction field $K$ of $D$ is a finite
extension of $\mathbb Q$, and $D$ is the integral closure of $\mathbb Z$ in $K$. Necessarily
$D$ is a one-dimensional integrally closed domain, i.e., a Dedekind domain. Moreover, a classical
result is that $D$ has a torsion class group.  In particular this means that every ideal $I$ in
$D$ has a power which is a principal ideal. This fact leads to the following theorem
about $\mathbb Z^+$.

\begin{thm}\label{bezout} Every finitely generated ideal of
$\mathbb Z^+$ is principal (a domain with this property is said to be a B\'ezout domain).
\end{thm}

\begin{proof} Let $I\subset \mathbb Z^+$ be generated by $t_1,...,t_m$. There is a finite field extension $L$
of $\mathbb Q$   containing all of these elements. The integral closure of $\mathbb Z$
in $L$, say $D$, is a Dedekind domain which contains $t_1,...,t_m$. Let $J$ be the ideal they
generate in $D$. By the discussion above, some power of $J$ is principal, say $J^n = (d)$.
We claim then that $d^{\frac{1}{n}}\mathbb{Z}^+ = I$. Let $i\in I$. We can write $i = \sum_j e_jt_j$, where the
$e_j\in \mathbb{Z}^+$. Again, there is a finite extension $T$ of $D$, with $T$ a Dedekind domain, such that 
$i\in I\cap T = (t_1,...,t_m)T = JT$. In a Dedekind domain there is unique factorization into prime ideals.
Write $JT = P_1^{a_1}\cdots P_k^{a_k}$. Then $J^nT = dT = P_1^{na_1}\cdots P_k^{na_k}$. By unique factorization,
it follows that $d^{\frac{1}{n}}T = P_1^{a_1}\cdots P_k^{a_k} = JT$, so that $i\in d^{\frac{1}{n}}\mathbb{Z}^+$.
\end{proof}

\bigskip
\bigskip
\section{Regular Sequences}

\medskip

We summarize some of the basic notions we will use to analyze $R^+$. 

\begin{definition}{\rm A sequence of elements $x_1,...,x_d$ in a ring $R$ is said to be
a \it regular sequence \rm if $rx_i\in (x_0,...,x_{i-1})$ implies that
$r\in (x_0,...,x_{i-1})$ for all $1\leq i\leq d-1$ (here we set $x_0 = 0$), and
$(x_1,...,x_d)R \ne R$.}
\end{definition}

Another definition we will need is the following.

\begin{definition} {\rm Let $(R,\m)$ be a Noetherian local ring of dimension $d$. Elements
$x_1,...,x_d$ are said to be a \it system of parameters \rm if the nilradical
of the ideal they generate is $\m$.
A Noetherian local ring is said to be \it Cohen-Macaulay \rm if some (equivalently) every
system of parameters forms a regular sequence.}
\end{definition}

We are aiming at a theorem which shows that in characteristic $p$, $R^+$ has regular sequences having
maximal length. It turns out that this is equivalent to being flat in an appropriate sense. The next proposition
proves this.

\medskip

\begin{prop} Let $A = k[[x_1,...,x_d]]$, where $k$ is a field of characteristic $p > 0$. Then the following two conditions
are equivalent:

(1) $x_1,...,x_d$ form a regular sequence on $A^+$ (by an abuse of language we say that $A^+$ is Cohen-Macaulay).
\newline
(2) $A^+$ is flat over $A$.
\end{prop}

\begin{proof}
We prove the equivalence. First assume (1).
If $A^+$ is not flat over $A$, choose $i\geq 1$ as large as possible so that
$\Tor_i^A(A/P, A^+)\ne 0$ for some prime $P$ in $A$. Such a choice is possible because $A$ is regular
and large Tors vanish. If $y_1,...,y_s$ is a maximal regular sequence
in $P$, then one can embed $A/P$ in $A/(y_1,...,y_s)$ with cokernel $C$. But since our assumption
forces $\Tor_{i+1}^R(C,A^+) =  0$ (as $C$ has a prime filtration), and $y_1,...,y_s$ form a
regular sequence on $A^+$, we obtain that $\Tor_i^A(A/P, A^+) = 0$, a contradiction.

To see that $y_1,...,y_s$ form a regular sequence, extend them to a system of parameters, and let
$B = k[[y_1,...,y_d]]$. Then $B^+ = A^+$, and our hypothesis says that the $y's$ form a regular
sequence.

Assume (2). Flat maps preserve regular sequence in general.
\end{proof}

Our method of studying regular sequences relies on local cohomology. We only need the description below.

 For $x\in R$, let $K^{\bullet}(x;R)$ denote the
complex $0\ra R\ra R_x\ra 0$, graded so that the degree $0$ piece of the complex is
$R$, and the degree $1$ is $R_x$. If $x_1,...,x_n\in R$, let $K^{\bullet}(x_1,x_2,...,x_n;R)$ denote
the complex $K^{\bullet}(x_1;R)\otimes_R ...\otimes_R K^{\bullet}(x_n;R)$, where in general recall that
if $(C^{\bullet},d_C)$ and $(D^{\bullet}, d_D)$ are complexes, then the tensor product
of these complexes, $(C\otimes_RD, \Delta)$ is by definition the complex whose ith
graded piece is $\sum_{j+k = i} C_j\otimes D_k$ and whose differential is determined
by the map from $ C_j\otimes D_k\ra (C_{j+1}\otimes D_k) \oplus (C_j\otimes D_{k+1})$
given by $\Delta(x\otimes y) = d_C(x)\otimes y + (-1)^k x\otimes d_D(y)$.

The modules in this complex, called the Koszul cohomology complex, are
$$0\ra R\ra \oplus\sum_i R_{x_i}\ra \oplus\sum_{i< j} R_{x_ix_j}\ra ...\ra R_{x_1x_2\cdots x_n}\ra 0$$
where the differentials are the natural maps induced from  localization, but with signs attached.
If $M$ is an $R$-module, we set $K^{\bullet}(x_1,x_2,...,x_n; M) = K^{\bullet}(x_1,x_2,...,x_n;R)\otimes_RM$.
We denote the cohomology of $K^{\bullet}(x_1,x_2,...,x_n; M)$ by $H^i_{I}(M)$, called
the $i$th local cohomology of $M$ with respect to $I = (x_1,...,x_d)$.
It is a fact that this module only depends on the ideal generated by the $x_i$ up to radical.
We summarize some useful information concerning these modules.

\begin{prop}\label{propbasechange} Let $R$ be a Noetherian ring, $I$ and ideal and $M$ and $R$-module.
Let $\phi: R\ra S$ be a homomorphism and let $N$ be an $S$-module.
\begin{enumerate}[\quad\rm (1)]
\item If $\phi$ is flat, then $H^j_I(M)\otimes_RS\cong H^j_{IS}(M\otimes_RS)$. In particular,
local cohomology commutes with localization and completion.
\item (Independence of Base) $H^j_I(N)\cong H^j_{IS}(N)$, where the first local cohomology
is computed over the base ring $R$.
\end{enumerate}
\end{prop}

\begin{proof} Choose generators $x_1,...,x_n$ of $I$. The first claim follows at once from the
fact that $K^{\bullet}(x_1,...,x_n;M)\otimes_RS = K^{\bullet}(\phi(x_1),...,\phi(x_n);M)\otimes_RS)$, and that
$S$ is flat over $R$, so that the cohomology of $K^{\bullet}(x_1,...,x_n;M)\otimes_RS$ is the cohomology
of $K^{\bullet}(x_1,...,x_n;M)$ tensored over $R$ with $S$.

The second claim follows from the fact that $$K^{\bullet}(x_1,...,x_n;N) = K^{\bullet}(x_1,...,x_n;R)\otimes_RN =
(K^{\bullet}(x_1,...,x_n;R)\otimes_RS)\otimes_SN$$
$$ = K^{\bullet}(\phi(x_1),...,\phi(x_n);S)\otimes_RN =
K^{\bullet}(\phi(x_1),...,\phi(x_n);N).$$
\end{proof}
\bigskip

\section{$R^+$ is Cohen-Macaulay in Positive Characteristic}

\medskip

Let $R$ be a commutative ring containing a field of characteristic $p>0$, let $I\subset R$
be an ideal, and let $R'$ be an $R$-algebra. The Frobenius ring homomorphism
$f:R'\stackrel{r\mapsto r^p}{\to}R'$ induces a map $f_*:H^i_I(R')\to H^i_I(R')$ on all
local cohomology modules of $R'$ called the action of the Frobenius on $H^i_I(R')$.
For an element $\alpha\in H^i_I(R')$ we denote $f_*(\alpha)$ by $\alpha^p$. This follows since
the Frobenius extends to localization of $R$ in the obvious way, and commutes with the
maps in the Koszul cohomology complex, which are simply signed natural maps.

The main result is that if $R$ is a local Noetherian domain which is a homomorphic image of a
Gorenstein local ring and has positive characteristic, then $R^+$ is Cohen-Macaulay in the
sense that every system of parameters of $R$ form a regular sequence in $R^+$. To prove this
result we use the proof given in \cite{HL}. The original proof, with slightly different assumptions, was given
in 1992 in \cite{HH}, as a result of developments from tight closure theory. Although tight closure has
now disappeared from the proof, it remains an integral part of the theory. A critical point is that we must find some
way of annihilating nonzero local cohomology classes. 
The next lemma is essentially the only way known to do this.

\begin{lemma} \label{element} Let $R$ be a commutative Noetherian domain containing a field of
characteristic $p>0$, let $K$ be the fraction field of $R$ and let $\overline K$ be the algebraic closure
of $K$. Let $I$ be an ideal of $R$ and let $\alpha\in H^i_I(R)$ be an element such that the elements
$\alpha, \alpha^p,\alpha^{p^2},\dots,\alpha^{p^t},\dots$ belong to a finitely generated $R$-submodule of $H^i_I(R)$.
There exists an $R$-subalgebra $R'$ of $\overline K$ (i.e. $R\subset R'\subset \overline K$) that is finite
as an $R$-module and such that the natural map $H^i_I(R)\to H^i_I(R')$ induced by the natural
inclusion $R\to R'$ sends $\alpha$ to 0.
\end{lemma}
\emph{Proof.} Let $A_t=\sum_{i=1}^{i=t}R\alpha^{p^i}$ be the $R$-submodule of $H^i_I(R)$ generated
by $\alpha,\alpha^p,\dots,\alpha^{p^t}$. The ascending chain $A_1\subset A_2\subset A_3\subset\dots$
stabilizes because $R$ is Noetherian and all $A_t$ sit inside a single finitely generated $R$-submodule
of $H^i_I(R)$. Hence $A_s=A_{s-1}$ for some $s$, i.e. $\alpha^{p^s}\in A_{s-1}$. Thus there exists an
equation $\alpha^{p^s}=r_1\alpha^{p^{s-1}}+r_2\alpha^{p^{s-2}}+\dots+r_{s-1}\alpha$ with $r_i\in R$
for all $i$. Let $T$ be a variable and let $g(T)=T^{p^s}-r_1T^{p^{s-1}}-r_2^{p^{s-2}}-\dots-r_{s-1}T$.
Clearly, $g(T)$ is a monic polynomial in $T$ with coefficients in $R$ and $g(\alpha)=0$.

Let $x_1,\dots, x_d\in R$ generate the ideal $I$. Recall that we can calculate the local cohomology
from the Koszul cohomology complex
$C^{\bullet}(R)$, 
$$0\to C^0(R)\to\dots \to C^{i-1}(R)\stackrel{d_{i-1}}{\to} C^i(R)\stackrel{d_i}{\to}
C^{i+1}(R)\to\dots\to C^d(R)\to 0$$
where $C^0(R)=R$ and $C^i(R)=\oplus_{1\leq j_1<\dots<j_{i}\leq d}R_{x_{j_1}\cdots x_{j_{i}}}$,
and $H^i_I(M)$ is the $i$th cohomology module of $C^{\bullet}(R)$.

Let $\tilde \alpha\in C^i(R)$ be a cycle (i.e. $d_i(\tilde \alpha)=0$) that represents $\alpha$.
The equality $g(\alpha)=0$ means that $g(\tilde \alpha)=d_{i-1}(\beta)$ for some $\beta\in C^{i-1}(R)$.
Since $C^{i-1}(R)=\oplus_{1\leq j_1<\dots<j_{i-1}\leq d}R_{x_{j_1}\cdots x_{j_{i-1}}}$, we may write
$\beta=(\frac{r_{j_1,\dots,j_{i-1}}}{x_{j_1}^{e_{1}}\cdots x_{j_{i-1}}^{e_{i-1}}})$
where $r_{j_1,\dots,j_{i-1}}\in R$, the integers $e_1,\dots, e_{i-1}$ are non-negative, and
$\frac{r_{j_1,\dots,j_{i-1}}}{x_{j_1}^{e_{1}}\cdots x_{j_{i-1}}^{e_{i-1}}}\in R_{x_{j_1}\cdots x_{j_{i-1}}}$.

Consider the equation $g(\frac{Z_{j_1,\dots,j_{i-1}}}{x_{j_1}^{e_{1}}\cdots x_{j_{i-1}}^{e_{i-1}}})
-\frac{r_{j_1,\dots,j_{i-1}}}{x_{j_1}^{e_{1}}\cdots x_{j_{i-1}}^{e_{i-1}}}=0$ where
$Z_{j_1,\dots,j_{i-1}}$ is a variable. Multiplying this equation by
$(x_{j_1}^{e_{1}}\cdots x_{j_{i-1}}^{e_{i-1}})^{p^s}$ produces a monic polynomial equation in
$Z_{j_1,\dots,j_{i-1}}$ with coefficients in $R$. Let $z_{j_1,\dots,j_{i-1}}\in \overline K$ be a root
of this equation and let $R''$ be the $R$-subalgebra of $\overline K$ generated by all the $z_{j_1,\dots,j_{i-1}}$s,
i.e. by the set $\{z_{j_1,\dots,j_{i-1}}|1\leq j_1<\dots<j_{i-1}\leq d\}$. Since each $z_{j_1,\dots,j_{i-1}}$
is integral over $R$ and there are finitely many $z_{j_1,\dots,j_{i-1}}$s, the $R$-algebra $R''$ is finite
as an $R$-module.

Let $\tilde{\tilde\alpha} =
(\frac{z_{j_1,\dots,j_{i-1}}}{x_{j_1}^{e_{1}}
\cdots x_{j_{i-1}}^{e_{i-1}}})\in C^{i-1}(R'')$. The natural inclusion $R\to R''$ makes $C^{\bullet}(R)$
into a subcomplex of $C^{\bullet}(R'')$ in a natural way, and we identify $\tilde \alpha\in C^i(R)$
and $\beta\in C^{i-1}(R)$ with their natural images in $C^i(R'')$ and $C^{i-1}(R'')$ respectively.
With this identification, $\tilde \alpha\in C^i(R'')$ is a cycle representing the image of $\alpha$
under the natural map $H^i_I(R)\to H^i_I(R'')$, and so is $\overline \alpha=\tilde\alpha-d_{i-1}(\tilde{\tilde\alpha})
\in C^i(R'')$. Since $g(\tilde{\tilde\alpha})=\beta$ and $g(\tilde \alpha)=d_{i-1}(\beta)$, we conclude that
$g(\overline \alpha)=0$. Let $\overline\alpha=(\rho_{j_1,\dots,j_i})$ where $\rho_{j_1,\dots,j_i}
\in R''_{x_{j_1}\cdots x_{j_i}}$. Each individual $\rho_{j_1,\dots,j_i}$ satisfies the equation
$g(\rho_{j_1,\dots,j_i})=0$. Since $g(T)$ is a monic polynomial in $T$ with coefficients in $R$,
each $\rho_{j_1,\dots,j_i}$ is an element of the fraction field of $R''$ that is integral over $R$.
Let $R'$ be obtained from $R''$ by adjoining all the $\rho_{j_1,\dots,j_i}$.
direct sum of all such copies of $R'$. This subcomplex is exact because its cohomology groups are the
cohomology groups of $R'$ with respect to the unit ideal. Since $\overline \alpha$ is a cycle and belongs to
this exact subcomplex, it is a boundary, hence it represents the zero element in $H^i_I(R')$. \qed

\medskip

We recall that for a Gorenstein local ring $A$ of dimension $n$, local duality says that
there is an isomorphism of functors $D({\rm Ext}^{n-i}_A(-, A))\cong H^i_{\mathfrak m}(-)$
on the category of finite $A$-modules, where $D={\rm Hom}_A(-,E)$ is the Matlis duality
functor (here $E$ is the injective hull of the residue field of $A$ in the category of
$A$-modules). 

\begin{thm}\label{module}\cite{HL}
Let $R$ be a commutative Noetherian local domain containing a field of characteristic $p>0$, let
$K$ be the fraction field of $R$ and let $\overline K$ be the algebraic closure of $K$. Assume $R$ is a surjective
image of a Gorenstein local ring $A$. Let $\mathfrak m$ be the maximal ideal of $R$. 
Let $i< \dim R$
be a non-negative integer. There is an $R$-subalgebra $R'$ of $\overline K$ (i.e. $R\subset R'\subset \overline K$)
that is finite as an $R$-module and such that the natural map $H^i_{\mathfrak m}(R)\to H^i_{\mathfrak m}(R')$
is the zero map.
\end{thm}

\emph{Proof.} The proof comes from \cite{HL}. Let $n=\dim A$ and let $N={\rm Ext}^{n-i}_A(R,A)$. 
Clearly $N$ is a finite $A$-module.

Let $d=\dim R$. We use induction on $d$. For $d=0$ there is nothing to prove, so we assume that
$d>0$ and the theorem proven for all smaller dimensions. Let $P\subset R$ be a non-maximal prime ideal.
We claim there exists an $R$-subalgebra $R^P$ of $\overline K$  such that
$R^P$ is a finite $R$-module and for every $R^P$-subalgebra $R^*$ of $\overline K$ (i.e. $R^P\subset R^*\subset \overline K$)
such that $R^*$ is a finite $R$-module, the image $\mathcal I\subset N$ of the natural map
${\rm Ext}^{n-i}_A(R^*,A)\to N$ induced by the natural inclusion $R\to R^*$ vanishes after localization at $P$,
i.e. $\mathcal I_P=0$. Indeed, let $d_P=\dim R/P$. Since $P$ is different from the maximal ideal,
$d_P>0$. As $R$ is a surjective image of a Gorenstein local ring, it is catenary, hence the dimension
of $R_{P}$ equals $d-d_P$, and $i<d$ implies $i-d_P<d-d_P=\dim R_{P}$. By the induction hypothesis
applied to the local ring $R_{P}$, 
there is an $R_{P}$-subalgebra $\tilde R$ of $\overline K$, which is finite as an $R_{P}$-module, such that
the natural map $H^{i-d_P}_{P}(R_P)\to H^{i-d_P}_{P}(\tilde R)$ is the zero map.
Let $\tilde R=R_{P}[z_{1},z_{2},\dots,z_{t}]$, where $z_{1},z_{2},\dots,z_{t}\in \overline K$ are
integral over $R_{P}$. Multiplying, if necessary, each $z_{j}$ by some element of $R\setminus P$,
we can assume that each $z_{j}$ is integral over $R$. We set $ R^P=R[z_{1},z_{2},\dots,z_{t}]$.
Clearly, $R^P$ is an $R'$-subalgebra of $\overline K$ that is finite as $R$-module.

Now let $R^*$ be both
an $R^P$-subalgebra of $\overline K$ (i.e. $R^P\subset R^*\subset \overline K$) and a finite $R$-module.
The natural inclusions $R\to R^P\to R^*$ induce natural maps
$${\rm Ext}^{n-i}_A(R^*,A)\to {\rm Ext}^{n-i}_A(R^P,A)\to N.$$ This implies that
$\mathcal I\subset \mathcal J$, where $\mathcal J$ is the image of the natural map
$\phi:{\rm Ext}^{n-i}_A(R^P,A)\to N$. Hence it is enough to prove that $\mathcal J_{P}=0$.
Localizing this map at $P$ we conclude that $J_P$ is the image of the natural map
$\phi_P:{\rm Ext}^{n-i}_{A_P}(\tilde R,A_P)\to {\rm Ext}^{n-i}_{A_P}(R_P,A_P)$ induced by the
natural inclusion $R_P\to \tilde R$ (by a slight abuse of language we identify the prime ideal
$P$ of $R$ with its full preimage in $A$). Let $D_P(-)={\rm Hom}_{A_P}(-, E_P)$ be the Matlis duality
functor in the category of $R_P$-modules, where $E_P$ is the injective hull of the residue field of
$R_P$ in the category of $R_P$-modules. Local duality implies that $D_P(\phi_P)$ is the natural map
$H^{i-d_P}_{P}(R_P)\to H^{i-d_P}_{P}(\tilde R)$ which is the zero map by construction
(note that $i-d_P=\dim A_P-(n-i)$). Since $\phi_P$ is a map between finite $R_P$-modules
and $D_P(\phi_P)=0$, it follows that $\phi_P=0$. This proves the claim.

Since $N$ is a finite $R$-module, the set of the associated primes of $N$ is finite.
Let $P_1,\dots,P_s$ be the associated primes of $N$ different from $\mathfrak m$. For each $j$
let $R^{P_j}$ be an $R$-subalgebra of $\overline K$ corresponding to $P_j$, whose existence is guaranteed
by the above claim. Let $\overline R=R[R^{P_1},\dots, R^{P_s}]$ be the compositum of all the $R^{P_j}$, $1\leq j\leq s$.
Clearly, $\overline R$ is an $R$-subalgebra of $\overline K$. Since each $R^{P_j}$
is a finite $R$-module, so is $\overline R$. Clearly, $\overline R$  contains every $R^{P_j}$. Hence the above
claim implies that $\mathcal I_{P_j}=0$ for every $j$, where $\mathcal I\subset N$ is the image of the
natural map ${\rm Ext}^{n-i}_A(\overline R,A)\to N$ induced by the natural inclusion $R\to\overline R$. It follows
that  not a single $P_j$ is an associated prime of $\mathcal I$. But $\mathcal I$ is a submodule of $N$,
and therefore every associated prime of $\mathcal I$ is an associated prime of $N$.
Since $P_1,\dots, P_s$ are all the associated primes of $N$ different from $\mathfrak m$, we conclude that if
$\mathcal I\ne 0$, then $\mathfrak m$ is the only associated prime of $\mathcal I$. Since $\mathcal I$,
being a submodule of a finite $R$-module $N$, is finite, and since $\mathfrak m$ is the only associated
prime of $\mathcal I$, we conclude that $\mathcal I$ is an $R$-module of finite length.

Writing the natural map ${\rm Ext}^{n-i}_A(\overline R,A)\to N$ as the composition of two maps
$${\rm Ext}^{n-i}_A(\overline R,A)\to \mathcal I\to N,$$ the first of which is surjective and the second injective,
and applying the Matlis duality functor $D$, we get that the natural map
$\varphi:H^i_{\mathfrak m}(R)\to H^i_{\mathfrak m}(\overline R)$ induced by the inclusion $R\to \overline R$ is the
composition of two maps $H^i_{\mathfrak m}(R)\to D(\mathcal I)\to H^i_{\mathfrak m}(\overline R)$, the first
of which is surjective and the second injective. This shows that the image of $\varphi$ is isomorphic to
$D(\mathcal I)$ which is an $R$-module of finite length since so is $\mathcal I$. In particular, the image
of $\varphi$ is a finitely generated $R$-module. Let $\alpha_1,\dots, \alpha_s\in
H^i_{\mathfrak m}(\overline R)$ generate Im$\varphi$.

The natural inclusion $R\to \overline R$ is compatible with the Frobenius homomorphism, i.e. with
the raising to the $p$th power on $R$ and $\overline R$. This implies that $\varphi$ is compatible
with the action of the Frobenius $f_*$ on $H^i_{\mathfrak m}(R)$ and $H^i_{\mathfrak m}(\overline R)$,
i.e. $\varphi(f_*(\alpha))=f_*(\varphi(\alpha))$ for every $\alpha\in H^i_{\mathfrak m}(R)$,
which, in turn, implies that Im$\varphi$ is an $f_*$-stable $R$-submodule of $H^i_{\mathfrak m}(\overline R)$,
i.e. $f_*(\alpha)\in {\rm Im}\varphi$ for every $\alpha\in {\rm Im}\varphi$. We finish the proof
by applying Lemma~\ref{element} to each element of  a finite generating set $\alpha_1,...,
\alpha_s$ of Im$\varphi$.
Applying Lemma~\ref{element} we obtain a
$\overline R$-subalgebra $R_j$ of $\overline K$ (i.e. $\overline R\subset R_j\subset \overline K$) such that
$R'_j$ is a finite $R$-module and the natural map $H^i_{\mathfrak m}(\overline R)\to H^i_{\mathfrak m}(R_j)$
sends $\alpha_j$ to zero. Let $R'=R[R_1,\dots, R_s]$ be the compositum of all the $R_j$.
Then $R'$ is an $R$-subalgebra of $\overline K$ and is a finite $R$-module since so is each $R_j$.
The natural map $H^i_{\mathfrak m}(\overline R)\to H^i_{\mathfrak m}(R')$ sends every $\alpha_j$ to zero,
hence it sends the entire Im$\varphi$ to zero. Thus the natural map
$H^i_{\mathfrak m}(R)\to H^i_{\mathfrak m}(R')$ is zero. \qed

\medskip

\begin{cor}\label{corcm}
Let $R$ be a commutative Noetherian local domain containing a field of characteristic $p>0$.
Assume that $R$ is a surjective image of a Gorenstein local ring. Then the following hold:

(a) $H^i_{\mathfrak m}(R^+)=0$ for all $i<\dim R$, where $\mathfrak m$ is the maximal ideal of $R$.

(b) Every system of parameters of $R$ is a regular sequence on $R^+$.
\end{cor}

\emph{Proof.} (a) $R^+$ is the direct limit of the finitely generated $R$-subalgebras $R'$, hence $H^i_{\mathfrak m}(R^+)=\varinjlim H^i_{\mathfrak m}(R')$.
But Theorem~\ref{module} implies that for each $R'$ there is $R''$ such that the map 
 $H^i_{\mathfrak m}(R')\to H^i_{\mathfrak m}(R'')$ in the inductive system is zero. Hence the limit is zero.

(b) Let $x_1,..., x_d$ be a system of parameters of $R$. We prove that $x_1,...,x_j$ is a regular
sequence on $R^+$ by induction on $j$. The case $j=1$ is clear, since $R^+$ is a domain.
Assume that $j>1$ and $x_1,\dots, x_{j-1}$ is a regular sequence on $R^+$. Set $I_t = (x_1,...,x_t)$.
The fact that $H^i_{\mathfrak m}(R^+)=0$ for all $i<d$ and the short exact sequences
$$0\to R^+/I_{t-1}R^+\stackrel{\rm x_t}{\longrightarrow} R^+/I_{t-1}R^+\to R^+/I_{t}R^+\to 0$$ for
$t\leq j-1$ imply by induction on $t$ that $H^q_{\mathfrak m}(R^+/(x_1,...,x_{t})R^+)=0$ for $q<d-t$.
In particular, $H^0_{\mathfrak m}(R^+/(x_1,...,x_{j-1})R^+)=0$ since $0<d-(j-1)$.
Hence $\mathfrak m$ is not an associated prime of $R^+/(x_1,...,x_{j-1})R^+$. This implies that the
only associated primes of $R^+/(x_1,...,x_{j-1})R^+$ are the minimal primes of $R/(x_1,...,x_{j-1})R$.
Indeed, if there is an embedded associated prime, say $P$, then $P$ is the maximal ideal of the ring
$R_P$ whose dimension is bigger than $j-1$ and $P$ is an associated prime of
$(R^+/(x_1,...,x_{j-1})R^+)_P=(R_P)^+/(x_1,...,x_{j-1})(R_P)^+$ which is impossible by the above.
Hence every element of $\mathfrak m$ not in any minimal prime of $R/(x_1,...,x_{j-1})R$, for example,
$x_j$, is a regular element on $R^+/(x_1,...,x_{j-1})R^+$.\qed

\medskip

\begin{disc}
{\rm It is important to understand the huge differences between characteristic $p$, characteristic $0$, and
mixed characteristic.
Suppose that $k$ has characteristic $0$. Let $A$ be a complete Noetherian local integrally closed domain with residue field $k$,
and fraction field $K$. 
If $L$ is any finite field extension of $K$ and $B$ is the integral closure of $A$
in $L$, then the reduced trace map\footnote{The \it reduced trace \rm is defined as $\frac{1}{n} Tr_{L/K}$ where
$L$ is a finite field extension of a field $K$, and $[L:K] = n$. This map fixes the ground field $K$. If
$R$ is an integrally closed domain with fraction field $K$ and $S$ is the integral closure of $R$ in $L$,
then the reduced trace sends $S$ to $R$ and fixes $R$.} gives a splitting of $A$ from $B$, i.e., $B\cong A\oplus N$ as
an $A$-module for some module $A$-module $N$. Then $H^i_{\m}(A)$ splits out of $H^i_{\m}(B)$, so
that the map $H^i_{\m}(A)\to H^i_{\m}(B)$ is \it never \rm zero unless $H^i_{\m}(A) = 0$. Thus the
exact opposite holds in characteristic $0$.}
\end{disc}

What happens in mixed characteristic is a great mystery. One of the great results in recent years was that
of Heitmann. He proved the following theorem:

\begin{thm} Let $(R,\m)$ be a complete three-dimensional Noetherian local integrally closed domain of
mixed characteristic $p\in \mathbb N$. (This means that $p\in \m$.) Then for all $n\geq 1$,
$p^{\frac {1}{n}}$ annihilates the local cohomology $H^2_{\m}(R^+)$.
\end{thm}

Heitmann's theorem is slightly stronger than this, but this is the essential result of \cite{He}. In characteristic
$p$ we know this local cohomology module is zero, but this is not known in mixed characteristic.

One can hope that if $R$ has mixed characteristic $p$, then $R^+/pR^+$ is Cohen-Macaulay. Of course, this
ring has positive characteristic. However, perhaps this is too much to hope for. The next best result would
be to conjecture that $R^+/\sqrt{pR^+}$ is Cohen-Macaulay, a question raised by Lyubeznik. He has a partial
result in this direction \cite{Ly}:

\begin{thm} Let $(R,\m)$ be a Noetherian local excellent domain of mixed characteristic $p$. Assume the
dimension of $R$ is at least $3$.  Set $\overline{R} = R/\sqrt{pR}$ and $\overline{R^+} = R^+/\sqrt{pR^+}$.
Then $H^1_{\m}(\overline{R^+}) = 0$, and every part of a system of parameters $a,b$ of $\overline{R}$
form a regular sequence on $\overline{R^+}$.
\end{thm}

\bigskip

Another interesting question is whether or not one must use inseparable extensions to trivialize local cohomology.
In fact, Anurag Singh \cite{Si} found the following nice trick to change inseparable elements to separable.

\begin{prop} Let $R$ be an excellent domain of characteristic $p > 0$, and let $I$ be an ideal of $R$. Suppose that
$z\in R$ is such that $z^q\in I^{[q]}$, where $ q = p^e$ is a power of $p$. Then there exists an integral domain
$S$, which is a module-finite separable extension of $R$, such that $z\in IS$.
\end{prop}
\bigskip

The point here is that there clearly a finite inseparable extension of $R$, say $T$, such that $z\in IT$. Simply
take $qth$ roots of the elements $a_j$ such that $z = \sum a_jx_j^q$ where $x_j\in I$. 

\begin{proof}  Write $z = \sum_{1\leq j\leq n} a_jx_j^q$ where $x_j\in I$ as above. Consider the equations for $2\leq i\leq n$,
$$ U_i^q + U_ix_1^q-a_i = 0.$$
These are monic separable equations and therefore have roots $u_i$ in a separable field extension of the fraction field
of $R$. Let $S$ be the integral closure of the ring $R[u_2,...,u_n]$. 
Since $R$ is excellent, $S$ is finite as an $R$-module. We claim that $z\in IS$. Set
$$u_1 = (z- \sum_{2\leq i\leq n} x_iu_i)/x_1.$$
Note that $u_1$ is an element of the fraction field of $S$. Taking $qth$ powers we see that
$$u_1^q = a_1 + \sum_{2\leq i\leq n} u_ix_i^q.$$
Therefore $u_1$ is integral over $S$. As $S$ is integrally closed, $u_1\in S$. This implies that $$z = \sum_{1\leq i\leq n} u_ix_i$$
and so $z\in IS$. \end{proof}

\begin{disc}{\rm  There is an interesting property pertaining to our main theorem. Suppose that
$(R,\m)$ is a complete local Noetherian domain of positive characteristic, and let $x_1,...,x_d$ form
a regular sequence. If $x_1,...,x_d$ is a system of parameters, and if $R$ is not Cohen-Macaulay, then
there is a non-trivial relation $r_1x_1+...+r_dx_d = 0$. Non-trivial means that it does not
come from the Koszul relations. Since $R^+$ is Cohen-Macaulay, we can trivialize this relation in $R^+$,
and therefore in some finite extension ring $S$ of $R$, $R\inc S\inc R^+$. But Theorem~\ref{module} does not say
whether or not there is a fixed finite extension ring $T$, $R\inc T\inc R^+$ in which all relations on all
parameters of $R$ become simultaneously trivial. Even if such a ring $T$ exists, this does not mean $T$ is
itself Cohen-Macaulay; new relations coming from elements of $T$ may be introduced. However, there is a finite
extension which simultaneously trivializes all relations on systems of parameters. This fact has been proved by
Melvin Hochster and Yongwei Yao \cite{HY}.}
\end{disc}

\section{Applications}
\bigskip

The existence of a big Cohen-Macaulay algebra has a great many applications. In some sense
it repairs the failure of a ring to be Cohen-Macaulay. Hochster proved and used the existence of big Cohen-Macaulay
modules (the word ``big" refers to the fact the modules may not be finitely generated) to prove many of the homological
conjectures. For a modern update, see \cite{Ho1}. In general, if you can prove a theorem in the
Cohen-Macaulay case, you should immediately try to use $R^+$ to try to prove it in general.
We give several examples of this phenomena in this section.  As examples, we will prove some
of the old homological conjectures using this approach; this is not new, but there are currently
a growing number of new homological conjectures, and it could be that characteristic $p$ methods
apply.

Of course, some of the homological conjectures deal directly with systems of parameters. These are
easy to prove once one has a Cohen-Macaulay module. For example, the next theorem gives the
monomial conjecture.

\begin{thm} Let $R$ be a local Noetherian ring of dimension $d$ and positive characteristic $p$.
Let $x_1,...,x_d$ be a system of parameters. Then for all $t\geq 1$, $(x_1\cdots x_d)^t$ is not
in the ideal generated by $x_1^{t+1},...,x_d^{t+1}$.
\end{thm}

\begin{proof} We use induction on the dimension $d$ of $R$. The case $d = 1$ is trivial.
Suppose by way of contradiction that $d > 1$ and  $(x_1\cdots x_d)^t\in (x_1^{t+1},...,x_d^{t+1})$.
This is preserved after completion, and is further preserved after moding out a minimal prime $P$
such that the dimension of the completion modulo $P$ is still $d$. After these operations,
the images of the elements $x_i$ still form a system of parameters as well.  Thus we may assume that
$R$ is a complete local domain. We apply Theorem~\ref{module} to conclude that $x_1,...,x_d$ is
a regular sequence in $R^+$. Write
$$(x_1\cdots x_d)^t = \sum_i s_ix_i^{t+1},$$
where $s_i\in R$. Then $x_d^t((x_1\cdots x_{d-1})^t - s_dx_d)\in (x_1^{t+1},...,x_{d-1}^{t+1})$.
Since the powers of the $x_i$ also form a regular sequence in $R^+$, we conclude that
$(x_1\cdots x_{d-1})^t - s_dx_d\in (x_1^{t+1},...,x_{d-1}^{t+1})R^+$. It follows that
there is a Noetherian complete local domain $S$ containing $R$ and module-finite over $R$
such that $(x_1\cdots x_{d-1})^t\in (x_1^{t+1},...,x_{d-1}^{t+1},x_d)S$. But now
$(x_1\cdots x_{d-1})^t$ is in the ideal $(x_1^{t+1},...,x_{d-1}^{t+1})$ in the ring $S/x_dS$,
which has dimension $d-1$. Our induction shows that this is impossible. \end{proof}

Next, we apply Theorem~\ref{module} it to various intersection theorems. One of the first such intersection conjectures
was: 

\begin{conj} Let $(R,\m)$ be a local Noetherian ring, and let $M,N$ be two finitely
generated nonzero $R$-modules such that $M\otimes_RN$ has finite length. Then
$$\dim N \leq \text{pd}_R(M).$$
\end{conj}

Of course there is nothing to prove if the projective dimension of $M$ is infinite.
We prove (see \cite{Ho}):

\begin{thm}\label{newintersection} Let $(R,\m)$ be a local Noetherian ring of positive prime characteristic $p$, and let $M,N$ be two finitely
generated nonzero $R$-modules such that $M\otimes_RN$ has finite length. Then
$$\dim N \leq \text{pd}_R(M).$$
\end{thm}

\begin{proof}
One can begin by making some easy reductions. These types of reduction are very good practice
in commutative algebra. First, note that the assumption that the tensor product has finite
length is equivalent to saying that $I+J$ is $\m$-primary, where $I = \text{Ann}(N)$ and
$J = \text{Ann}(M)$. Then we can choose a prime $P$ containing $I$ such that $\dim(R/P) =
\dim (N)$, and observe that we can replace $N$ by $R/P$ without loss of generality.
It is more difficult to change $M$, since the property of being finite projective dimension
does not allow many changes.  

Let's just suppose for a moment that $R/P$ is Cohen-Macaulay. Since $P+J$ is $\m$-primary,
we can always choose $x_1,...,x_d\in J$ whose images in $B = R/P$ form a system of parameters
(and thus are a regular sequence in $R/P$). If the projective dimension of $M$
is smaller than $d = \dim(R/P)$, then  $Tor_d^{R}(B/(x_1,...,x_d)B, M) = 0$.  
Notice that $Tor_0^{R}(B, M) \ne 0$. We claim by induction that for $0\leq i\leq d$,
$Tor_i^{R}(B/(x_1,...,x_i)B, M) \ne 0$. When $i=d$ we arrive at a contradiction.
Suppose we know this for $i < d$.  Set $B _i = B/(x_1,...,x_i)$. The short exact sequence
$0\to B_i\to B_i\to B_{i+1}\to 0$ obtained by multiplication
by $x_{i+1}$ on $B_i$ induces a map of Tors when tensored with $M$. Since all $x_i$ kill $M$,
we obtain a surjection of $Tor_{i+1}^{R}(B_{i+1}, M)$ onto $Tor_i^{R}(B_i, M)$.
This finishes the induction.

Of course, we don't know that $R/P$ is Cohen-Macaulay, and in general it won't be. But now
suppose that we are in positive characteristic. We can first complete $R$ before beginning
the proof. Now $R/P$ is a complete local domain, and $S = (R/P)^+$ is Cohen-Macaulay in
the sense that $x_1,..,x_d$ form a regular sequence in this ring. The same proof works
verbatium, provided we know that $S\otimes_RM\ne 0$. But this is easy; it is even nonzero
after passing to the residue field of $S$.  \end{proof}

As a corollary, we get a favorite of the old Chicago school of commutative algebra, 
the zero-divisor conjecture (now a theorem):

\bigskip

\begin{thm} Let $(R,\m)$ be a Noetherian local ring of characteristic $p$, and let
$M$ be a nonzero finitely generated $R$-module having finite projective dimension. If
$x$ is a non-zerodivisor on $M$, then $x$ is a non-zerodivisor on $R$.
\end{thm}

\begin{proof} This proof is taken from \cite{PS}.  First observe that the statement of the theorem is equivalent to saying
that every associated prime of $R$ is contained in an associated prime of $M$. We induct on
the dimension of $M$ to prove this statement. If $\dim M = 0$, then the only associated
prime of $M$ is $\m$, which clearly contains every prime of $R$. Hence we may assume that
$\dim M > 0$. Let $P\in \text{Ass}(R)$. First suppose that there is a prime $Q\in \text{Supp}(M), Q\ne \m,$
such that $P\inc Q$. Then we can change the ring to $R_Q$ and the module to $M_Q$. By induction, $P_Q$ is
contained in an associated prime of $M_Q$, so lifting back gives us that $P$ is in an associated
prime of $M$. We have reduced to the case in which $R/P\otimes_RM$ has finite length. By 
Theorem~\ref{newintersection}, $\dim (R/P)\leq \text{pd}_R(M) =  \text{depth}(R) - \text{depth}(M)$. Since
$P$ is associated to $R$,  $\dim (R/P)\geq \text{depth}(R)$ (exercise). It follows that
the depth of $M$ is $0$, and hence the maximal ideal is associated to $M$ (and contains $P$).
\end{proof} 

\medskip

For a completely different type of application, we consider an old result of Grothendieck's concerning
when the punctured spectrum of a local ring is connected.
There is a beautiful proof of Grothendieck's result in all characteristics due to Brodmann and Rung \cite{BR}. The main
point here is that if $R$ is Cohen-Macaulay and the $x_i$ are parameters, then there is a very easy
proof. It turns out that one can always assume that the $x_i$ are parameters, and then the proof
of the Cohen-Macaulay case directly generalizes to one in characteristic $p$ using $R^+$. 
The exact statement is:

\begin{thm} Let $(R,\m)$ be a complete local Noetherian domain of dimension $d$, and let
$x_1,...,x_k\in \m$, where $k\leq d-2$. Then the punctured spectrum of $R/(x_1,...,x_k)$ is
connected.
\end{thm}

\begin{proof} We take the proof from \cite{HH}.
First assume that the $x_i$ are parameters.
Let $I$ and $J$ give a disconnection of the punctured spectrum
of $R/(x_1,...,x_k)$. Choose elements $u+v$, $y+z$ which together with the $x_i$ form parameters
such that $u,y\in I$ and $v,z\in J$. Modulo $(x_1,...,x_k)$ one has the relation
$y(u+v) - u(y+z) = 0$. Since the parameters form a regular sequence in $R^+$, we obtain that
$y\in (y+z,x_1,...,x_k)R^+$. Similarly, $z\in (y+z,x_1,...,x_k)R^+$. Write $y = c(y+z)$ modulo
$(x_1,...,x_k)$ and $z = d(y+z)$ modulo $(x_1,...,x_k)$. Then $(1-c-d)(y+z)\in (x_1,...,x_k)R^+$
so that $1-c-d\in (x_1,...,x_k)R^+$. At least one of $c$ or $d$ is a unit in $R^+$, say $c$. But then
$y$ is not a zerodivisor modulo $(x_1,...,x_k)R^+$ and this implies that $J\inc (x_1,...,x_k)R^+$.
Then $I$ must be primary to $mR^+$ which is a contradiction since the height of $I$ is too small.

It remain to reduce to the case in which the $x_i$ are parameters.

We claim that any $k$-elements are up to radical in an ideal generated by $k$-elements which are
parameters.
 The key point is to prove this for $k = 1$. Suppose that $x = x_1$ is given. If $x$
already has height one we are done. If $x$ is nilpotent, choose $y$ to be any parameter in $I$. So assume
that $x$ is not in every minimal prime. For $n\gg 0$, $0:x^n = 0:x^{n+1}$, and changing $x$ to $x^n$,
we obtain that $x$ is not a zero divisor on $R/(0:x)$. Then there is an element $s\in 0:x$ such that
$y = x +s$ has height one, and we may multiply $s$ by a general element of $I$ to obtain that
$y\in I$. But $xy = x^2$ so that $x$ is nilpotent on $(y)$. Inductively choose
$y_1,...,y_{k-1}$ which are parameters such that the ideal they generate contains $x_1,...,x_{k-1}$ up
to radical. Replace $R$ by $R/(y_1,...,y_{k-1})$, and repeat the $k = 1$ step.

Now if
$I$ and $J$ disconnect the punctured spectrum of $R/(x_1,...,x_k)$ choose any ideals $I'$ and $J'$
of height at least $k$, not primary to the maximal ideal such that $I'$ contains $I$ up to radical
and $J'$ contains $J$ up to radical. Choose parameters in $I'\cap J'$ such that the $x_i$ are in the
radical $K$ of these parameters. Then $K + I$ and $K + J$ disconnect the punctured spectrum of $R/K$.
\end{proof}

\end{document}